\documentclass[10pt, reqno]{amsart}

\usepackage{amsmath,amssymb,amsthm}

\usepackage{tikz}

\usepackage{caption}
\usepackage{graphicx}
\usepackage{graphics}
\usepackage{algorithm}
\usepackage{algorithmicx}
\usepackage{algpseudocode}

\newtheorem*{thm}{Theorem}
\newtheorem{lemma}{Lemma}

\begin{document}

\title[]{A set of points on the sphere \\with small Riesz energy}

\author[]{Stefan Steinerberger}

 \address{Department of Mathematics and Department of Applied Mathematics, University of Washington, Seattle}
 \email{steinerb@uw.edu}

\begin{abstract} 
We construct a set of points $\left\{x_1, \dots, x_n\right\} \subset \mathbb{S}^2$ such that
$$ \sum_{i \neq j} \frac{1}{\|x_i - x_j\|^2} \leq \frac{n^2 \log{n}}{4} + cn^2 + \mathcal{O}(n^{11/6} \log{n}),$$
where the constant $c \sim -0.085768\dots$ is given in closed form and matches the constant that was 
conjectured by Brauchart-Hardin-Saff to be optimal. The point set is motivated by the crystallization conjecture
and consists of pieces of the hexagonal lattice projected onto the sphere in a tightly interlocked way.
\end{abstract}

\maketitle

\section{Introduction and Results}
We study the minimal Riesz energy of $n$ points on the unit sphere $\mathbb{S}^2$. This problem is classical, we refer to the book of Borodachov--Hardin--Saff \cite{boro}. One particularly natural case is the interaction energy $1/\|x-y\|^2$ which is known to be critical, the strength of the singularity matches the dimension leading to a logarithmic correction.
Brauchart--Hardin--Saff \cite[Conjecture 5]{bhs} conjecture the asymptotic minimal energy to scale as
$$ \min_{x_1, \dots, x_n \in \mathbb{S}^2} \quad  \sum_{i \neq j} \frac{1}{\|x_i - x_j\|^2} =  \frac{n^2 \log{n}}{4} + c_* n^2 + \mathcal{O}(1),$$
where $c_* = -0.0857684\dots$ has a closed form expression involving a coefficient arising in the Laurent series expansion of the Hurwitz zeta function. The bound
$$ \min_{x_1, \dots, x_n \in \mathbb{S}^2} \quad  \sum_{i \neq j} \frac{1}{\|x_i - x_j\|^2} \leq  \frac{n^2 \log{n}}{4} + c n^2 + o(n^2)$$
has been established for
\begin{align*}
c &= \gamma/4  \sim0.144\dots \qquad && \text{(\mbox{Alishahi--Zamani } \cite{ali})}\\
c &= \frac{1}{4} \left(\frac{3}{2} - \log(2\pi) + \gamma\right)  \sim0.0598\dots \qquad && \text{(\mbox{de la Torre--Marzo } \cite{de})} \\
c &= \frac{1}{2}\left(\gamma - \frac{1}{2} - \frac{\log{2}}{3} \right)  \sim-0.0769 \dots \qquad && \text{(\mbox{L\'opez-G\'omez } \cite{lo}),}
\end{align*}
where $\gamma$ is the Euler-Mascheroni constant.
All these bounds are constructive:
Alishahi and Zamani use a determinantal point process (the spherical ensemble); de la Torre and Marzo adapt an approach pioneered by Armentano--Beltran--Shub \cite{arm} and use the roots of random elliptic polynomials. The recent result of L\'opez-G\'omez, which very nearly matches the conjectured asymptotic rate, uses the diamond ensemble of Beltran--Etayo \cite{beltran}. It is widely believed that the constant conjectured by Brauchart--Hardin--Saff is indeed the correct one and this is, in a philosophical sense, implied by the crystallization conjecture \cite{crys}. Unconditionally, the best lower bound is due to Brauchart \cite{brauchart} who used an ingenious idea (see also Bilyk--Brauchart \cite{brauchart2}) to show
$$   \min_{x_1, \dots, x_n \in \mathbb{S}^2} \quad  \sum_{i \neq j} \frac{1}{\|x_i - x_j\|^2} \geq  \frac{n^2 \log{n}}{4} +  \frac{\gamma - 1}{4} n^2 + o(n^2),$$
where
$(\gamma - 1)/4 \sim - 0.10569 \dots$
nearly matches the conjectured constant. We construct an explicit set of points whose $2-$Riesz energy behaves as conjectured.

\begin{thm} There exists a set $\left\{x_1, \dots, x_n\right\} \subset \mathbb{S}^2$ such that
$$ \max_{1 \leq i \leq n} \sum_{j=1 \atop j \neq i}^n \frac{1}{\|x_i - x_j\|^2} \leq  \frac{n \log{n}}{4} +  \log\left(  \frac{(2\pi)^{3/4}}{3^{1/4}}  \frac{e^{\gamma/2}}{  \Gamma(1/3)^{3/2}}   \right)  n + \mathcal{O}(n^{5/6}\log{n}),$$
where $\gamma$ is the Euler-Mascheroni constant.
\end{thm}

 The constant
 $$  \log\left(  \frac{(2\pi)^{3/4}}{3^{1/4}}  \frac{e^{\gamma/2}}{  \Gamma(1/3)^{3/2}}   \right)  = -0.085768410300902483655821715047\dots$$
 is the same as the one obtained by Brauchart--Hardin--Saff: we use the first Kronecker limit formula to arrive at an expression involving the Dedekind eta function which is then evaluated using the Chowla-Selberg formula. One could simply refer to the abstract quantities arising in the Laurent series around the pole at $s=1$, that would lead to the alternative expression. 
 \begin{center}
\begin{figure}[h!]
\begin{tikzpicture}
 \node at (6,0) {\includegraphics[width=0.46\textwidth]{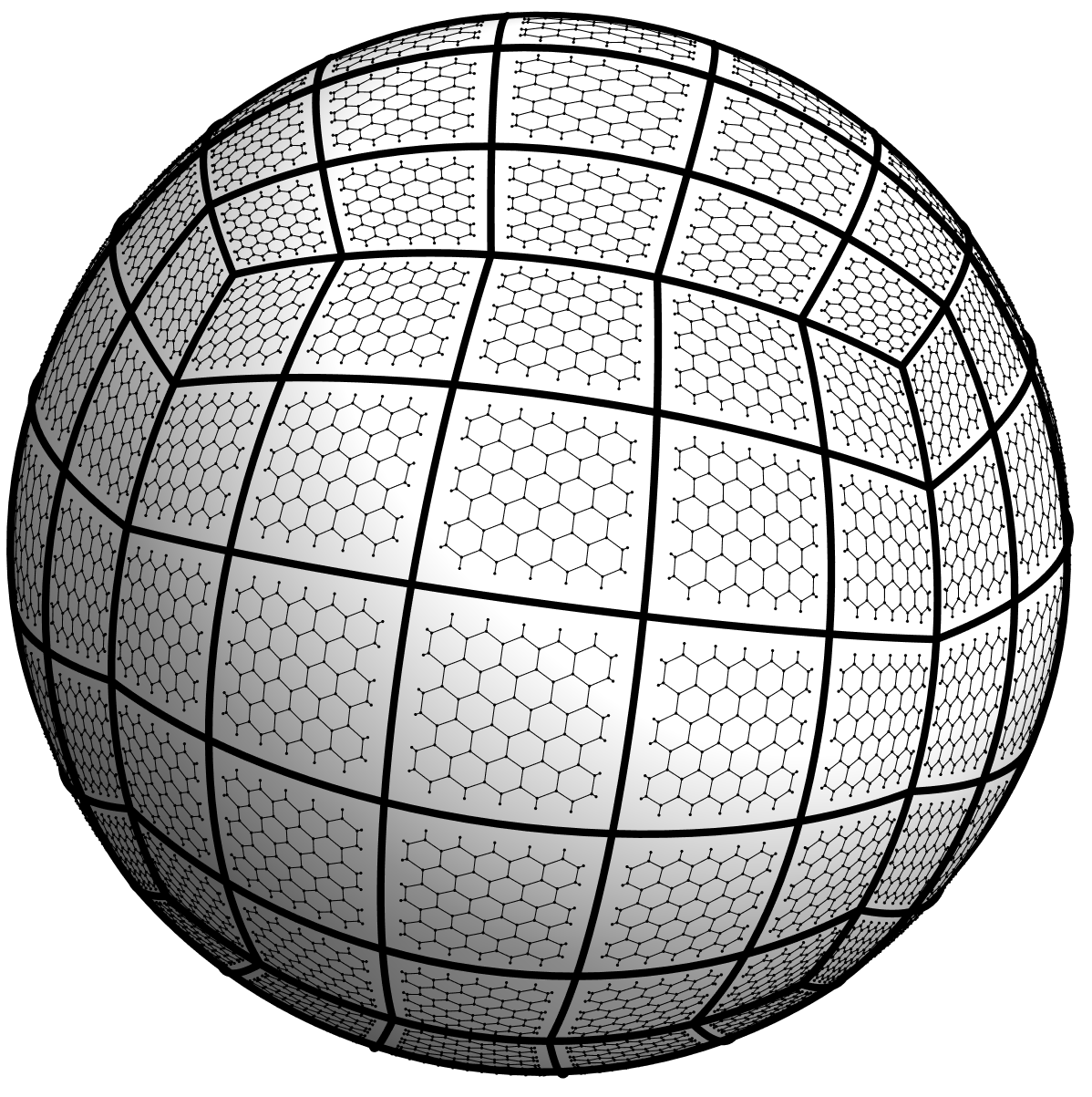}};
\end{tikzpicture}
\caption{Sketch of the set: $\mathbb{S}^2$ is partitioned into similarly sized convex cells, each containing the projection of a hexagonal lattice.}
\end{figure}
\end{center}

We did not optimize the proof to minimize the error term $\mathcal{O}(n^{5/6} \log{n})$. Some improvement is possible (and sketched in the proof) but in its present form there is no reason to assume that this construction would lead to a good term past the second order without a refined analysis of the boundary of the quadrilaterals.\newpage

\textbf{Idea behind the construction.} The idea behind the construction is fairly simple.
\begin{enumerate}
\item The main term $n^2 \log{n}/4$ is only a little bit larger than the interesting scale $n^2$. The leading order term is easy to predict, one has a lot of wiggle room (and, indeed, the error term $\mathcal{O}(n^{5/6} \log{n})$ is quite large).
\item The crystallization conjecture suggests that optimal configurations will be `mostly' locally hexagonal. The significant decay of $\|x-y\|^{-2}$ suggests that interactions across large distances should be easier to control. This suggests placing locally hexagonal structures and gluing them together.
\item This strategy has already been successfully implemented by  Habicht and van der Waerden \cite{habicht}  in 1951, they proved the existence of a set of points $\left\{x_1, \dots, x_n\right\} \subset \mathbb{S}^2$ such that
$$ \min_{i \neq j} \|x_i - x_j\| \geq   \left( \frac{8 \pi}{\sqrt{3}} \right)^{1/2} \frac{1}{\sqrt{n}} - \mathcal{O}\left(\frac{1}{n^{2/3}}\right).$$
\end{enumerate}

It is tempting to believe that any set of points that is nearly maximally separated in the sense of Habicht and van der Waerden should
have to be locally hexagonal in most places and should then be very nearly optimal for all Riesz energies; this conjecture is arguably out of
reach, however, it does suggest adapting the construction of Habicht and van der Waerden to find sets of points with small energy. Our construction gets the correct leading two orders for $s=2$ and the leading order for $s = \infty$, it is reasonable to assume that it does well in the range $2 < s < \infty$.

\section{Description of the Points}
This section is heavily inspired by the construction proposed by Habicht and van der Waerden \cite{habicht} and describes how to derive the relevant half of their result from simple geometric considerations (they prove a lower bound on minimal separation; we describe their example showing that this rate can be attained; that the rate is best possible is related to sphere packing problem in $\mathbb{R}^2$ and not relevant here).\\

\textbf{Construction.} The first step is a decomposition of $\mathbb{S}^2$ into `local cells' as follows: first inscribe a standard cube
around $\mathbb{S}^2$ such that each of the six sides touches the sphere. We then partition each of the 6 sides into 
$\sim n^{2/3}$ squares (see Fig. 1).
\begin{center}
\begin{figure}[h!]
\begin{tikzpicture}
\node at (0,0) {\includegraphics[width=0.3\textwidth]{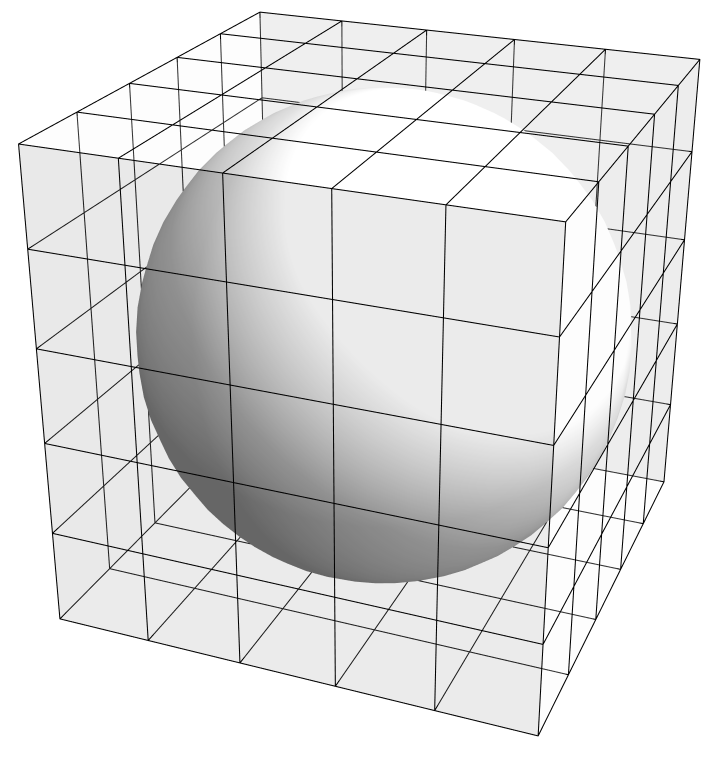}};
\draw[<->] (-2,0.2) -- (-2,0.6);
\node at (-2.8, 0.4) {$\sim n^{-1/3}$};
\node at (5.5,0) {\includegraphics[width=0.3\textwidth]{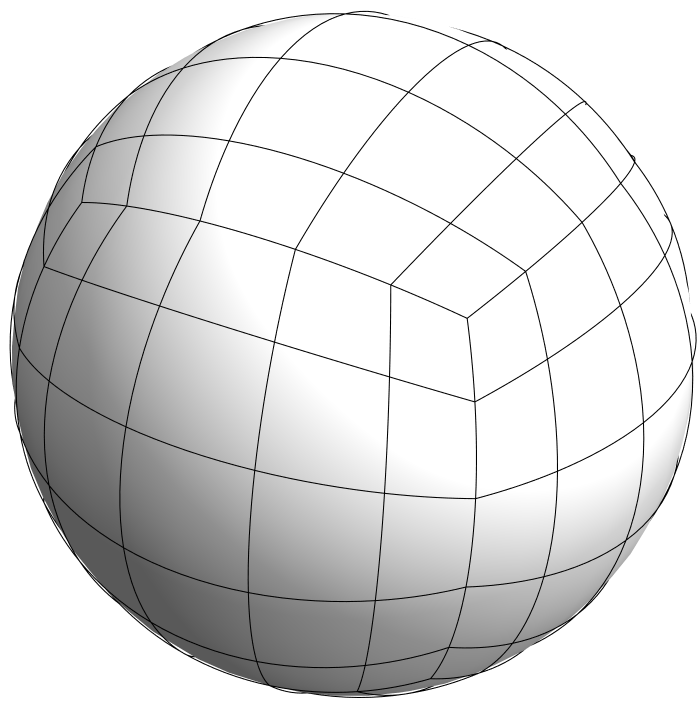}};
\end{tikzpicture}
\caption{The enclosing cube containing the squares (left) and the projection of these squares onto $\mathbb{S}^2$ (right)}
\end{figure}
\end{center}

We will use the convention of using precise constants when these are required and $\sim $ when the precise value of the leading order constant will not be important; thus, for example, decomposing each of the six squares into $\sim n^{2/3}$ smaller squares means that we can decompose them into $\left\lfloor n^{2/3} \right\rfloor$ or $\left\lfloor 2 n^{2/3} \right\rfloor$ squares or anything else at that scale.
 After that decomposition, we project each of these squares onto $\mathbb{S}^2$.  This gives a partition of $\mathbb{S}^2$ into $\sim n^{2/3}$ spherical quadrilaterals. We note that these quadrilaterals
do not have the same area (there is a distortion that is also clearly visible in Fig. 1), however, their areas are comparable up to universal constants. More precisely, we have a decomposition of $\mathbb{S}^2$ into $\sim n^{2/3}$ spherical quadrilaterals such that
 each quadrilateral has area $\sim n^{-2/3}$ and spatial dimensions $\sim n^{-1/3} \times n^{-1/3}$. For each spherical quadrilateral
 we remove a $100 \cdot n^{-1/2}$ neighborhood of its boundary: since $n^{-1/2} \ll n^{-1/3}$ this is a very small part of the quadrilateral, we
 may think of it as a tiny area around the boundary. 
 For each spherical quadrilateral $Q$, pick its barycenter $b \in Q$, construct the tangent plane to the sphere in $b$ and consider this tangent plane to be a copy of $\mathbb{R}^2$ with $b$ being the origin. We then consider the standard hexagonal grid $\Lambda$ in $\mathbb{R}^2$, rescaled by a factor of
$  \sqrt{ 8 \pi} \cdot 3^{-1/4} \cdot n^{-1/2}$
and project this hexagonal grid onto $Q$.  That factor is chosen so that the Voronoi cell of the rescaled hexagonal lattice, has area
$$ 2\sqrt{3} \left[ \frac{1}{2} \left( \frac{8 \pi}{\sqrt{3}} \right)^{1/2} \frac{1}{\sqrt{n}} \right]^2 = \frac{4\pi}{n}$$
which is aligned with $n$ `evenly distributed' points on $\mathbb{S}^2$ (having surface area $4\pi$).
We project these points onto the sphere and keep all the points that land inside $Q$ and have distance at least $100 \cdot n^{-1/2}$ from the boundary of $Q$ (see Fig. 3).

\begin{center}
\begin{figure}[h!]
\begin{tikzpicture}
\node at (0,0) {\includegraphics[width=0.8\textwidth]{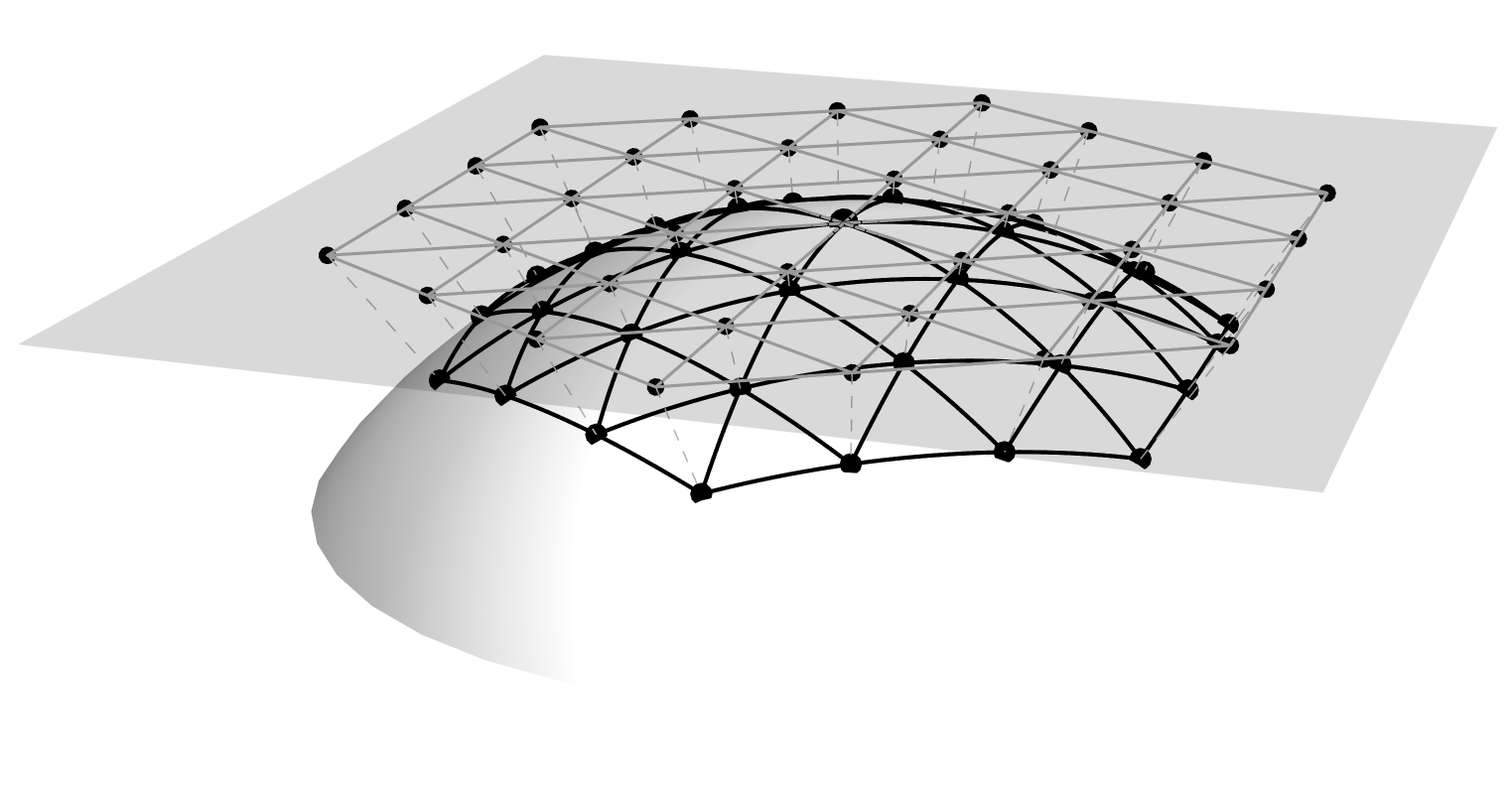}};
\end{tikzpicture}
\vspace{-20pt}
\caption{The hexagonal lattice in a tangent plane being projected onto the sphere; the point where the tangent plane touches $\mathbb{S}^2$ lies in a quadrilateral $Q$ and we keep all the projected points that are at least distance $100/\sqrt{n}$ from the boundary of $\partial Q$.}
\end{figure}
\end{center}

\textbf{Properties.} This leads to a set of approximately $n$ approximately hexagonal points on the unit sphere $\mathbb{S}^2$.  We make this precise by first estimating how many points in this construction will land in the spherical quadrilateral $Q$ which we assume to have surface area $|Q|$. Since $Q$ has spatial dimensions $n^{-1/3} \times n^{-1/3}$, the local area distortion when projecting onto the tangent plane $T$ is of order $1 + \mathcal{O}(n^{-2/3})$. Therefore the area of the projection $\pi(Q)$ of the quadrilateral from the sphere onto the tangent plane is 
$$ |\pi(Q)| = (1+ \mathcal{O}(n^{-2/3})) \cdot |Q| = |Q| + \mathcal{O}(n^{-4/3}).$$
It remains to determine the number of (rescaled) hexagonal points that are contained in $\pi(Q)$.  Since $\pi(Q)$ is convex and has diameter and inradius comparable up to a constant (it is `round' as opposed to `elongated'), we can do the usual elementary trick of bounding the number by perturbing the boundary. The boundary of $\pi(Q)$ has length $\sim n^{-1/3}$ and we are removing a $100\cdot n^{-1/2}$ neighborhood from it which leads to an area correction of size $\mathcal{O}(n^{-1/3} n^{-1/2}) = \mathcal{O}(n^{-5/6})$. 
The fundamental cell of the rescaled hexagonal lattice has area $4\pi/n$ and thus the number of points in the spherical quadrilateral $Q$ is
 \begin{align*}
 \frac{n}{4 \pi} |\pi(Q)| + \mathcal{O}(n^{-1/3} n^{-1/2} n) = \frac{n}{4\pi} |Q| + \mathcal{O}(n^{1/6}).
 \end{align*}
 There are $\sim n^{2/3}$ quadrilaterals whose areas sum up to $4\pi$ and thus we end up with a total of
 $n + \mathcal{O}(n^{5/6})$ points. It is worth mentioning, at this point, that the error term will have a negative sign: we will typically
 have $n - c \cdot n^{5/6}$ points at this stage of the construction; these  $n + \mathcal{O}(n^{5/6})$  points have large separation
 $$ \min_{i \neq j} \|x_i - x_j\| \geq   \left( \frac{8 \pi}{\sqrt{3}} \right)^{1/2} \frac{1}{\sqrt{n}} - \mathcal{O}(n^{-7/6}).$$
 The leading order term in this estimate is exactly the distance between two points from the lattice in the tangent plane; we do not need to concern ourselves with points from different `patches' since these are at least $100 \cdot n^{-1/2}$ separated which is much larger than the desired bound. For two points from the same patch, the effect of projecting two points at distance $\delta$ from part of a tangent plane with diameter $\sim n^{-1/3}$ leads to a distance distortion of no more than $\sim n^{-2/3}$. Since the distance cannot shrink more than by a factor of $(1- c \cdot n^{2/3})$ and the initial distance is $\sim n^{-1/2}$, we end up with an error of size no larger than $\mathcal{O}(n^{-1/2} n^{-2/3})= \mathcal{O}(n^{-7/6})$. However, we do not quite have $n$ points, we only have $n - c n^{5/6}$ points. We therefore have to rerun the argument with $n + 2 c n^{5/6}$ as the initial parameter and arrive at a set with at least $n$ points with minimal distance
$$ \left( \frac{8 \pi}{\sqrt{3}} \right)^{1/2} \frac{1}{\sqrt{n + 2 c n^{5/6}}}  - \mathcal{O}(n^{-7/6}) = \left( \frac{8 \pi}{\sqrt{3}} \right)^{1/2} \frac{1}{\sqrt{n}} - \mathcal{O}(n^{-2/3}).$$
 This is, somewhat rewritten, the argument of Habicht -- van der Waerden concerning the largest possible minimal distance for any set of $n$ points on $\mathbb{S}^2$.

\section{The Main Lemma} 
The purpose of this section is to provide the main lemma that describes the behavior of the Riesz 2-energy with respect to the hexagonal lattice $\Lambda$ in $\mathbb{R}^2$.  The statement can be read as having two parts: the first is that, for some $A \in \mathbb{R}$, we have the asymptotic expansion
$$  \sum_{k \in \Lambda \atop 0 < \|k\| < X} \frac{1}{\|k\|^2} = \frac{4\pi}{\sqrt{3}} \log{X} + A + \mathcal{O}(X^{-1}).$$
This is easy to establish via elementary techniques. The second part of the argument identifies the value of $A$ and relies on some ideas that are, while standard in the literature, decidedly not robust (the Kronecker Limit Formula for the asymptotic behavior of the Epstein zeta function near the critical singularity).

\begin{lemma} Let $\Lambda$ denote the standard hexagonal lattice in $\mathbb{R}^2$. Then
$$ \sum_{k \in \Lambda \atop 0 < \|k\| < X} \frac{1}{\|k\|^2} = \frac{4\pi}{\sqrt{3}} \log\left( \frac{e^{\gamma}}{\sqrt{3} |\eta(e^{\pi i/3})|^2}X \right) + \mathcal{O}(X^{-1}).$$
where $\eta$ is the Dedekind eta function.
\end{lemma}
\begin{proof}
We use $N(r)$ to denote the number of nonzero lattice points in a closed disk of radius $r$, that is $N(r) = \#\{k \in \Lambda \setminus \{0\} : 0< \|k\| \le r\}.$
We bound this using the standard argument coming from the basic estimates for the Gauss circle problem: the area of the fundamental cell is $\sqrt{3}/2$ showing that
$$ N(r) = \frac{r^2 \pi}{\sqrt{3}/2} + \mathcal{O}(r) = \frac{2 \pi}{\sqrt{3}} r^2 + E(r),$$
where the error term $E(r) = \mathcal{O}(r)$ comes from the boundary. Rewriting the sum as a Riemann-Stieltjes integral and using integration by parts, we get
$$ \sum_{k \in \Lambda \atop 0 < \|k\| < X} \frac{1}{\|k\|^2} = \int_{1}^X \frac{1}{r^2} \, dN(r) = \frac{N(X)}{X^2} + 2 \int_1^X \frac{N(r)}{r^3} \, dr.$$
Using the asymptotics for $N(r)$, we get
\begin{align*}
\sum_{k \in \Lambda \atop 0 < \|k\| < X} \frac{1}{\|k\|^2} &= \frac{2 \pi}{\sqrt{3}} + \mathcal{O}(X^{-1}) + 2 \int_1^X \frac{2 \pi}{\sqrt{3}} \frac{1}{r}dr + 2\int_1^X \frac{E(r)}{r^3} dr \\
&= \frac{4 \pi}{\sqrt{3}} \log{X} + \frac{2\pi}{\sqrt{3}} + 2\int_1^X \frac{E(r)}{r^3} dr.
\end{align*}
  Since $E(r) = \mathcal{O}(r)$, we have
  $$  \int_1^X \frac{E(r)}{r^3} dr =  \int_1^{\infty} \frac{E(r)}{r^3} dr -  \int_X^{\infty} \frac{E(r)}{r^3} dr =   \int_1^{\infty} \frac{E(r)}{r^3} dr  + \mathcal{O}(X^{-1}).$$
Altogether, we arrive at
$$ \sum_{k \in \Lambda \atop 0 < \|k\| < X} \frac{1}{\|k\|^2} = \frac{4 \pi}{\sqrt{3}} \log{X} + \frac{2\pi}{\sqrt{3}} + 2\int_1^{\infty} \frac{E(r)}{r^3} dr + \mathcal{O}(X^{-1}). \qquad (\diamond)$$
This is the desired asymptotic expansion.  It remains to identify the constant in question.
We use that the sum is very nearly convergent. Indeed, for any $s > 1$,
$$ \sum_{k \in \Lambda \atop 0 < \|k\| < X} \frac{1}{\|k\|^{2s}} < \infty.$$
Arguing in exactly the same way, we see that, for $s>1$ and $X \rightarrow \infty$,
\begin{align*}
 \sum_{k \in \Lambda \atop k \neq 0} \frac{1}{\|k\|^{2s}} &=  2s \int_1^{\infty} \frac{N(r)}{r^{2s+1}} dr = 2s \int_1^{\infty} \frac{1}{r^{2s+1}}\left( \frac{2 \pi}{\sqrt{3}} r^2 + E(r) \right) dr \\
 &=  \frac{2 \pi}{\sqrt{3}} \frac{1}{s-1} + \frac{2\pi}{\sqrt{3}}  + 2s \int_1^{\infty} \frac{E(r)}{r^{2s + 1}} dr. \qquad \qquad (\diamond \diamond)
 \end{align*}
The last two terms converge exactly to the constant in question when $s \rightarrow 1^+$. We now expand the expression on the left using the first Kronecker limit formula (see, for example, Lang \cite[Chapter 20]{lang}). If, for $\tau \in \mathbb{C}$ and $s > 1$,
$$ E(\tau, s) = \sum_{(m,n) \in \mathbb{Z}^2 \atop (m,n) \neq (0,0)} \frac{(\Im \tau)^s}{\| m \tau + n\|^{2s}}$$
then there is an asymptotic expansion as $s \rightarrow 1^+$ given by
$$ E(\tau, s) = \frac{\pi}{s-1} + 2\pi \gamma - 2\pi \log{2} - 2\pi \log\left( \sqrt{\Im (\tau)} |\eta(\tau)|^2 \right) + \mathcal{O}(s-1),$$
where $\eta$ is the Dedekind eta function. The hexagonal lattice is generated by $(1,0)$ and  $\tau = e^{\pi i/3}$ whose imaginary part is $\Im \tau = \sqrt{3}/2$. We need a version of the formula without the $(\Im \tau)^s$ term and use that, for $s \rightarrow 1^+$
$$ (\Im \tau)^s = \left( \frac{\sqrt{3}}{2}\right)^s =  \frac{\sqrt{3}}{2} - \frac{\sqrt{3}}{4} \log\left(\frac{4}{3}\right) (s-1) + \mathcal{O}((s-1)^2)$$
to write
\begin{align*}
 \sum_{(m,n) \in \mathbb{Z}^2 \atop (m,n) \neq (0,0)} \frac{(\sqrt{3}/2)^s}{\| m \tau + n\|^{2s}} &= \sum_{(m,n) \in \mathbb{Z}^2 \atop (m,n)  \neq (0,0)} \frac{\sqrt{3}/2}{\| m \tau + n\|^{2s}} + \mathcal{O}(s-1) \\
 &- \frac{\sqrt{3}}{4} \log\left(\frac{4}{3}\right) (s-1)  \sum_{(m,n) \in \mathbb{Z}^2 \atop (m,n) \neq (0,0)} \frac{1}{\| m \tau + n\|^{2s}}.
\end{align*}
The last term can be simplified a bit further.
Since, ignoring lower order terms, one more application of the Kronecker Limit Formula implies that
$$ \sum_{(m,n) \in \mathbb{Z}^2 \atop (m,n) \neq (0,0)} \frac{1}{\| m \tau + n\|^{2s}} = \frac{2}{\sqrt{3}} \sum_{(m,n) \in \mathbb{Z}^2 \atop (m,n) \neq (0,0)} \frac{\sqrt{3}/2}{\| m \tau + n\|^{2s}} + \mathcal{O}(s-1) = \frac{2}{\sqrt{3}} \frac{\pi}{s-1} + \mathcal{O}(1),$$
we arrive at
$$ \frac{\sqrt{3}}{4} \log\left(\frac{4}{3}\right) (s-1)  \sum_{(m,n) \in \mathbb{Z}^2 \atop (m,n) \neq (0,0)} \frac{1}{\| m \tau + n\|^{2s}} = \frac{\pi}{2} \log\left(\frac{4}{3}\right) + \mathcal{O}(s-1)
 $$
 and can simplify the previous expression as
 \begin{align*}
 \sum_{(m,n) \in \mathbb{Z}^2 \atop (m,n) \neq (0,0)} \frac{(\sqrt{3}/2)^s}{\| m \tau + n\|^{2s}} &= \sum_{(m,n) \in \mathbb{Z}^2 \atop (m,n)  \neq (0,0)} \frac{\sqrt{3}/2}{\| m \tau + n\|^{2s}}  - \frac{\pi}{2}\log\left(\frac{4}{3}\right) + \mathcal{O}(s-1).  
\end{align*}
 
Altogether, we obtain
\begin{align*}
 \sum_{k \in \Lambda \atop \|k\| \neq 0} \frac{1}{\|k\|^{2s}} &= \frac{2}{\sqrt{3}} \sum_{(m,n) \in \mathbb{Z}^2 \atop (m,n) \neq (0,0)} \frac{(\Im \tau)^s}{\| m \tau + n\|^{2s}} + \frac{\pi}{\sqrt{3}}\log\left(\frac{4}{3}\right) + \mathcal{O}(s-1) \\
 &= \frac{2 \pi}{\sqrt{3}} \frac{1}{s-1} + \frac{4\pi}{\sqrt{3}} \gamma - \frac{4\pi}{\sqrt{3}} \log{2} + \frac{\pi}{\sqrt{3}}\log\left(\frac{4}{3}\right) \\
 &- \frac{4\pi}{\sqrt{3}} \log\left( \frac{3^{1/4}}{\sqrt{2}}   |\eta(e^{\pi i/3})|^2 \right) + \mathcal{O}(s-1).
\end{align*}
Since, again for $s \rightarrow 1^+$,
$$ 2s \int_1^{\infty} \frac{E(r)}{r^{2s + 1}} dr = 2 \int_1^{\infty} \frac{E(r)}{r^{3}} dr + \mathcal{O}(s-1),$$
we may compare this asymptotic expansion to $(\diamond \diamond)$ and deduce
$$ \frac{2\pi}{\sqrt{3}}  + 2 \int_1^{\infty} \frac{E(r)}{r^{3}} dr + \mathcal{O}(s-1) = \frac{4\pi}{\sqrt{3}} \log\left( \frac{e^{\gamma}}{\sqrt{3} |\eta(e^{\pi i/3})|^2} \right).$$
Plugging this back into $(\diamond)$ gives the desired expansion.
\end{proof}

The argument only used the `trivial' lattice point counting estimate
$$\#\{k \in \Lambda \setminus \{0\} : 0< \|k\| \le r\} = \frac{2 \pi}{\sqrt{3}} r^2 + \mathcal{O}(r).$$
Refined estimates are available and would lead to a smaller error term; for the sake of keeping the argument as self-contained as possible, we have not pursued this.

\section{Local contributions: the near field} 
We write the sum as
$$ \sum_{i=2}^{n} \frac{1}{\|x_1 - x_i\|^2} = \sum_{i=2 \atop \|x_1 - x_i\| \leq \varepsilon}^{n} \frac{1}{\|x_1 - x_i\|^2} +
 \sum_{i=2 \atop \|x_1 - x_i\| \geq \varepsilon}^{n} \frac{1}{\|x_1 - x_i\|^2}$$
 and we refer to the first term as the `local' (or: near-field) contribution and the second term as the `global' (or: far-field) contribution. As usual, these two regimes require two very different techniques, \S 4 deals with the near field while \S 5 will deal with the far field. We may assume without loss of generality that $x_1 = (0,0,1)$ is in the north pole of the unit sphere and that we are summing over all other points that are contained in a spherical cap of size $\varepsilon$ (which will later be chosen to be $\varepsilon = n^{-1/10}$). The main idea is simple: projecting these nearby points to the tangent plane going through $x_1 = (0,0,1)$ decreases distances; the projected points should look approximately like the hexagonal lattice and we should be able to apply the main Lemma.  It remains to control the error terms. Let $y_1, \dots, y_{\ell}$ denote the subset of all the points $x_2, \dots, x_n$ contained in the spherical cap and let $z_1, \dots, z_{\ell}$ denote their orthogonal projections onto the tangent plane going through the north pole. Then
 $$ \|x_1 - z_i\| \leq \|x_1 - y_i\| \qquad \mbox{and} \qquad \sum_{i=1}^{\ell} \frac{1}{\|x_1 - y_i\|^2} \leq  \sum_{i=1}^{\ell} \frac{1}{\|x_1 - z_i\|^2}.$$
Assuming $\varepsilon = o(1)$, the projections $z_i$ are still, to first order, very well approximated by hexagonal lattices. We will bound the sum from above by using a smaller hexagonal mesh. The points on the sphere have minimal separation 
$$ \min_{i \neq j} \|x_i - x_j\| \geq   \left( \frac{8 \pi}{\sqrt{3}} \right)^{1/2} \frac{1}{\sqrt{n}} - \mathcal{O}(n^{-2/3}).$$
The maximal shrinking of distances of projections occurs at distance $\sim \varepsilon$ from the north pole, the relevant factor being
$$ \frac{1}{\sqrt{1 + f'(\varepsilon)^2}} \qquad \mbox{where} \qquad f(x) = \sqrt{1-x^2}.$$
This factor is $1 - \varepsilon^2/2 - \mathcal{O}(\varepsilon^4) \geq 1- \varepsilon^2$. 
Thus
$$ \min_{i \neq j} \|z_i - z_j\| \geq  (1-\varepsilon^2) \left[  \left( \frac{8 \pi}{\sqrt{3}} \right)^{1/2} \frac{1}{\sqrt{n}} - \mathcal{O}(n^{-2/3}) \right].$$
This allows us to compare the sum of inverse distances from $x_1$ to the projected points to the sum of the inverse distances of a rescaling of the standard unit hexagonal lattice $\Lambda$ in $\mathbb{R}^2$ by a factor of
 \begin{align*}
  \lambda &= (1- \varepsilon^2)\left[ \left( \frac{8 \pi}{\sqrt{3}} \right)^{1/2} \frac{1}{\sqrt{n}} -  \mathcal{O}(n^{-2/3})\right] \\
  &= \left( \frac{8 \pi}{\sqrt{3}} \right)^{1/2} \frac{1}{\sqrt{n}} - \mathcal{O}(n^{-2/3}) - \mathcal{O}(\varepsilon^2 n^{-1/2}).
  \end{align*}
Using this estimate, we arrive at
  \begin{align*}
  \sum_{i=1}^{\ell} \frac{1}{\|x_1 - z_i\|^2} &\leq  \sum_{k \in \Lambda \atop 0 < \|  k\| \leq \frac{\varepsilon}{\lambda}} \frac{1}{\|\lambda k\|^2} = \frac{1}{\lambda^2}   \sum_{k \in \Lambda \atop 0 < \|  k\| \leq \frac{\varepsilon}{\lambda}} \frac{1}{\| k\|^2}.
  \end{align*}
  
  \begin{center}
\begin{figure}[h!]
\begin{tikzpicture}
\node at (0,0) {\includegraphics[width=\textwidth]{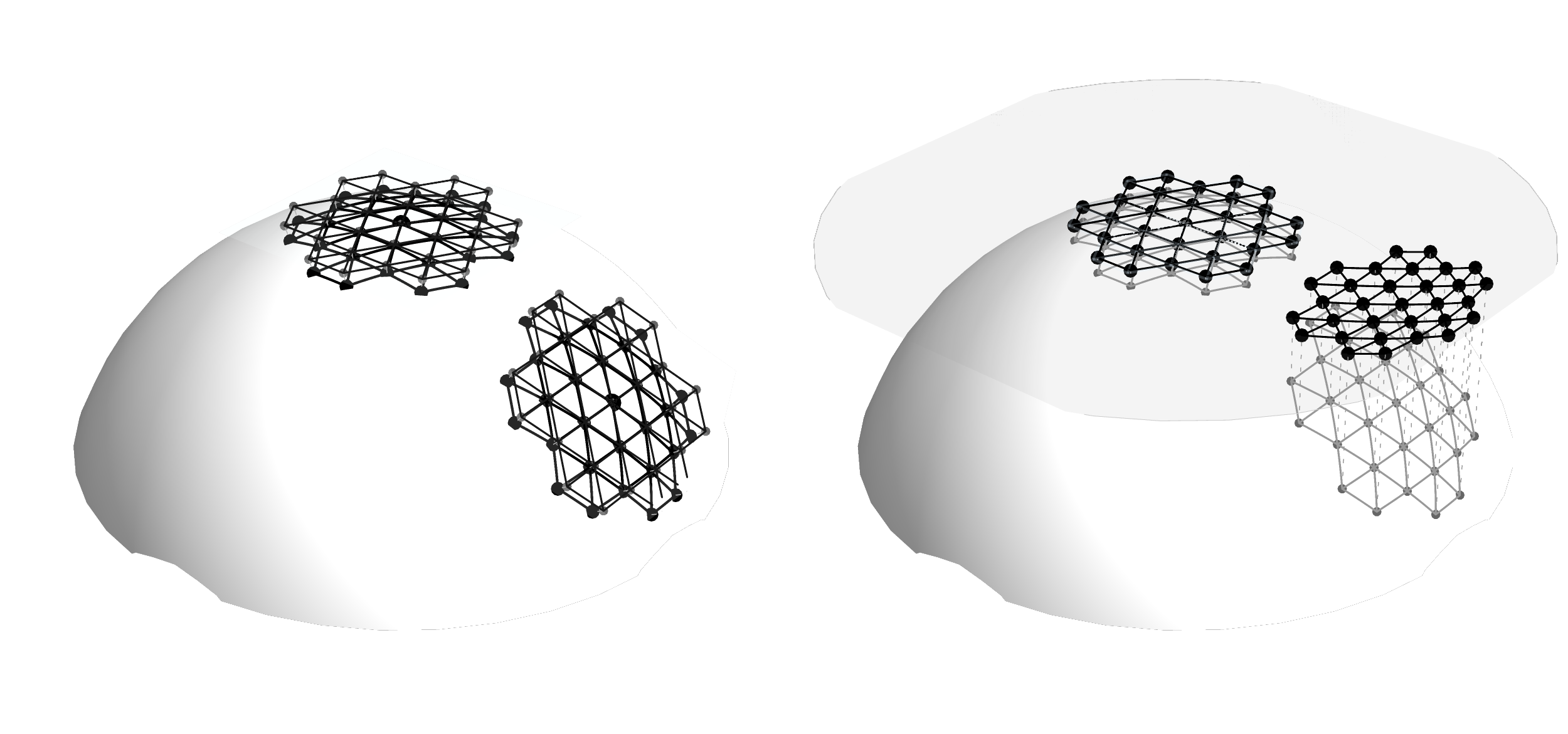}};
\end{tikzpicture}
\vspace{-20pt}
\caption{The hexagonal patches are constructed locally; when projecting them with respect to a far-away tangent plane, distortions will occur that move points closer to each other.}
\end{figure}
\end{center}
  \vspace{-20pt}
  
    We quickly note that slightly more refined estimates are possible: the distortion is smaller for points that are closer to the north pole which is also where the summands $1/\|k\|^2$ are larger; however, as mentioned above, we are not optimizing the proof for the smallest possible error term (which, if pursued as a main goal, would arguably also require a more precise analysis of the placement of the spherical quadrilaterals). We continue with an analysis of the sum. Note that
  $$ \frac{1}{\lambda^2} = \frac{\sqrt{3}}{8\pi} n + \mathcal{O}(n^{5/6}) + \mathcal{O}( \varepsilon^2 n ).$$
  As for the sum, we use the Lemma to deduce that, provided $\varepsilon \gg n^{-1/2}$,
  \begin{align*}
   \sum_{k \in \Lambda \atop 0 < \|  k\| \leq \frac{\varepsilon}{\lambda}} \frac{1}{\| k\|^2} &=  \frac{4\pi}{\sqrt{3}} \log\left( \frac{e^{\gamma}}{\sqrt{3} |\eta(e^{\pi i/3})|^2} \frac{\varepsilon}{\lambda} \right) + \mathcal{O}(\varepsilon^{-1} n^{-1/2}).
  \end{align*}
We have
  $$ \log \left( \frac{1}{\lambda}\right) =  \log \left( \frac{ 3^{1/4} \sqrt{n}}{ \sqrt{8 \pi}}\right) + \mathcal{O}(n^{-1/6}) + \mathcal{O}(\varepsilon^2),$$
  we can also write this as, assuming $\varepsilon \gg n^{-1/2}$,
    \begin{align*}
   \sum_{k \in \Lambda \atop 0 < \|  k\| \leq \frac{\varepsilon}{\lambda}} \frac{1}{\| k\|^2} &=  \frac{4\pi}{\sqrt{3}} \log\left(\varepsilon \frac{e^{\gamma}}{\sqrt{3} |\eta(e^{\pi i/3})|^2}  \frac{ 3^{1/4} \sqrt{n}}{ \sqrt{8 \pi}}\right) \\
   &+ \mathcal{O}(n^{-1/6} + \varepsilon^2 + \varepsilon^{-1} n^{-1/2}).
  \end{align*} 
  and multiplying this with $1/\lambda^2$, we arrive at the estimate
  \begin{align*}
  \sum_{i=2 \atop \|x_1 - x_i\| \leq \varepsilon}^{n} \frac{1}{\|x_1 - x_i\|^2} &\leq   \frac{n}{2} \log\left(\varepsilon \frac{e^{\gamma}}{\sqrt{3} |\eta(e^{\pi i/3})|^2}  \frac{ 3^{1/4} \sqrt{n}}{ \sqrt{8 \pi}}\right) \\
  &+ \mathcal{O}(n^{5/6} \log{n} +  \varepsilon^2 n \log{n} + \varepsilon^{-1} n^{1/2}).
  \end{align*}

\section{Global contributions: the far field} 
The points have nearly optimal separation in the sense of Habicht and van der Waerden; this means that we can associate to each point a `fundamental' hexagonal cell whose smallest area is very close to $4\pi/n$ (up to a lower order term). Since we are far away from the singularity, the point evaluation $\|x_1 - x_i\|^{-2}$ does not vary too much if $x_i$ is replaced by another point inside the hexagonal cell associated to $x_i$. This allows us to bound the sum by the integral and suggests that
$$  \sum_{i=2 \atop \|x_1 - x_i\| \geq \varepsilon}^{n} \frac{1}{\|x_1 - x_i\|^2} \sim  \frac{n}{4\pi} \int_{\mathbb{S}^2} \frac{1_{\|x_1-y\| \geq \varepsilon}}{\|x_1-y\|^2} dy.$$
This will now be made precise. The integral has a simple closed form: assuming $x_1$ to be in the north pole and using spherical coordinates, we have
$$ \| x- y\|^2 = 2 - 2 \cos{\theta}.$$
The condition $\|x - y\| \geq \varepsilon$ then leads to the cutoff $1- \cos{\theta_{\varepsilon}} = \varepsilon^2/2$ and, using spherical coordinates, the
area element $\sin{\theta} d\theta$ and the substitution $u = 1 - \cos{\theta}$
\begin{align*}
 \int_{\mathbb{S}^2} \frac{1_{\|x_1-y\| \geq \varepsilon}}{\|x_1-y\|^2} dy  = \int_0^{2\pi} \int_{\theta_{\varepsilon}}^{\pi} \frac{\sin{\theta}}{2 - 2 \cos{\theta}} d\theta = \pi \int_{\varepsilon^2/2}^{2} \frac{du}{u} = 2 \pi \log\left(\frac{2}{\varepsilon} \right).
\end{align*}
It remains to understand the induced error. 
We use the fact that we can associate to each point $x_i \in \mathbb{S}^2$ a hexagonal cell $H(x_i) \subset \mathbb{S}^2$ with width at least
$$ \delta =  \min_{i \neq j} \|x_i - x_j\| \geq   \left( \frac{8 \pi}{\sqrt{3}} \right)^{1/2} \frac{1}{\sqrt{n}} - \mathcal{O}(n^{-2/3}).$$
If everything was `perfectly equi-spaced', then each such cell would have area $4 \pi /n$. However, due to the correction factor of size $1 + \mathcal{O}(n^{-1/6})$,
the area may be distorted by up to $(1 + \mathcal{O}(n^{-1/6}))^2 = 1 + \mathcal{O}(n^{-1/6})$. Each such cell has at least area 
$$ |H(x_i)| \geq (1 + \mathcal{O}(n^{-1/6})) \cdot \frac{4 \pi}{n}.$$
Any such fundamental cell has size $\sim n^{-1/2} \times n^{-1/2}$. The function $\|x_1-y\|^{-2}$ can vary by at most $n^{-1/2}/\varepsilon^3$ over each cell. Therefore
\begin{align*}
\frac{1}{\|x_1- x_i\|^2}  &= \frac{1}{|H(x_i)|} \int_{H(x_i)}  \frac{1}{\|x_1- x_i\|^2} dy \\
&\leq  \frac{1}{|H(x_i)|} \int_{H(x_i)}  \left( \frac{1}{\|x_1- y\|^2} + \mathcal{O}(n^{-1/2} \varepsilon^{-3})\right) ~dy\\
&=  \mathcal{O}(n^{-1/2} \varepsilon^{-3})  + \frac{1}{|H(x_i)|} \int_{H(x_i)}  \frac{1}{\|x_1- y\|^2} ~dy\\
&\leq  \mathcal{O}(n^{-1/2} \varepsilon^{-3})  + \left(\frac{n}{4\pi} + \mathcal{O}(n^{5/6}) \right) \int_{H(x_i)}  \frac{1}{\|x_1- y\|^2} ~dy.
\end{align*}
Summing over at most $n$ terms gives 
\begin{align*}
 \sum_{i=2 \atop \|x_1 - x_i\| \geq \varepsilon}^{n} \frac{1}{\|x_1 - x_i\|^2} &\leq  \frac{n}{4\pi}  \int_{\mathbb{S}^2} \frac{1_{\|x-y\| \geq \varepsilon}}{\|x-y\|^2} dy + \mathcal{O}(n^{1/2} \varepsilon^{-3}) + \mathcal{O}(n^{5/6} \log(2/\varepsilon)) \\
 &= \frac{n}{2} \log\left( \frac{2}{\varepsilon} \right) + \mathcal{O}(n^{1/2} \varepsilon^{-3}) + \mathcal{O}(n^{5/6} \log(2/\varepsilon)).
 \end{align*}
Once again, the estimate is pessimistic: the error term $n^{-1/2} \varepsilon^{-3}$ is only accurate for points $x_i$ satisfying
 distance $\|x_1- x_i\| \sim \varepsilon$. Most points are further away.
 
 \section{Finishing the proof}
 Collecting all the terms, we obtain
 \begin{align*}
 \sum_{i=2}^{n} \frac{1}{\|x_1 - x_i\|^2} &\leq   \frac{n}{2} \log\left(\varepsilon \frac{e^{\gamma}}{\sqrt{3} |\eta(e^{\pi i/3})|^2}  \frac{ 3^{1/4} \sqrt{n}}{ \sqrt{8 \pi}}\right) \\
  &+ \mathcal{O}(n^{5/6} \log{n} + n \varepsilon^2 \log{n} + \varepsilon^{-1} n^{1/2}) \\
  &+ \frac{n}{2} \log\left( \frac{2}{\varepsilon} \right) + \mathcal{O}(n^{1/2} \varepsilon^{-3}) + \mathcal{O}(n^{5/6} \log(2/\varepsilon)).
 \end{align*}
 Extracting the $\sqrt{n}$ term inside the logarithm and combining the remaining two terms and keeping track of the dominant error terms, we arrive at the expression
  \begin{align*}
 \sum_{i=2}^{n} \frac{1}{\|x_1 - x_i\|^2} &\leq \frac{n \log{n}}{4} +    \log\left(  \frac{1}{3^{1/8}} \frac{ 1}{ (2 \pi)^{1/4}} \frac{e^{\gamma/2}}{  |\eta(e^{\pi i/3})|}   \right) n \\
  &  + \mathcal{O}(n^{5/6} \log{n} + \varepsilon^2 n \log{n} + n^{1/2} \varepsilon^{-3}+ n^{5/6} \log(2/\varepsilon)).
 \end{align*}
 We simplify the constant by appealing to the Chowla-Selberg formula \cite{chowla, chowla2} 
 $$  |\eta(e^{\pi i/3})| = \frac{3^{1/8}}{2\pi} \Gamma \left( \frac{1}{3} \right)^{3/2}.$$
Setting $\varepsilon = n^{-1/10}$ results in the error term $\mathcal{O}(n^{5/6} \log{n})$.

\end{document}